\documentclass{article}
\usepackage{amsmath,amsthm,amsopn,amstext,amscd,amsfonts,amssymb}
\usepackage{natbib}

\def \eig {\mathop{\rm eig}\nolimits}
\def \tr {\mathop{\rm tr}\nolimits}
\def \re {\mathop{\rm Re}\nolimits}

\def \Vol {\mathop{\rm Vol}\nolimits}
\def \etr {\mathop{\rm etr}\nolimits}
\def \diag {\mathop{\rm diag}\nolimits}
\def \build#1#2#3{\mathrel{\mathop{#1}\limits^{#2}_{#3}}}

\renewenvironment{abstract}
                 {\vspace{6pt}
                  \begin{center}
                  \begin{minipage}{5in}
                  \centerline{\textbf{Abstract}}
                  \noindent\ignorespaces
                 }
                 {\end{minipage}\end{center}}
\newtheorem{thm}{\textbf{Theorem}}[section]
\newtheorem{cor}{\textbf{Corollary}}[section]
\newtheorem{lem}{\textbf{Lemma}}[section]
\theoremstyle{definition}

\newtheorem{rem}{\textbf{Remark}}[section]

\title{\Large \textbf{Central matricvariate and matrix multivariate $T$ distributions}}
\author{
  \textbf{Jos\'e A. D\'{\i}az-Garc\'{\i}a} \thanks{Corresponding author\newline
    }\\
  {\normalsize Department of Statistics and Computation} \\
  {\normalsize 25350 Buenavista, Saltillo, Coahuila, Mexico} \\
  {\normalsize E-mail: jadiaz@uaaan.mx} \\[2ex]
  \textbf{Ram\'on Guti\'errez J\'aimez} \\
  {\normalsize Department of Statistics and O.R.} \\
  {\normalsize University of Granada} \\
  {\normalsize Granada 18071, Spain}\\
  {\normalsize E-mail: rgjaimez@ugr.es}\\
}
\date{}
\begin{document}
\maketitle

\begin{abstract}
Several distributions are studied, simultaneously in the real, complex, quaternion and octonion
cases. Specifically, these are the central, nonsingular matricvariate and matrix multivariate
$T$ and beta type II distributions and the joint density of the singular values are obtained
for real normed division algebras.
\end{abstract}

\section{Introduction}\label{sec1}

The complex case has renewed interest in multivariate analysis in diverse areas of science and
technology, see \citet{me:91}, \citet{rva:05a} and \citet{mdm:06}, among many others. Moreover,
diverse works involving multivariate analysis have appeared in the context of the quaternion
case, see \citet{bh:00}, \citet{f:05}, \citet{lx:09}, among others. However, with respect to
the octonion case, only a few, theoretical results have been published, see \citet{f:05} and
\citet{k:65}. This lack of widespread interest may be, because as stated by \citet{b:02},
\textit{...there is still no proof that the octonions are useful for understanding the real
world}. Nevertheless, for the sake of completeness, we include results in the octonion case as
conjectures.

Using definitions, properties and notation from abstract algebra, we propose a unified approach
that enables the simultaneous study of the distribution of a random matrix in the real,
complex, quaternion and octonion cases, which is generically termed the, distribution of a
random matrix for real normed division algebras.

In particular, the matricvariate $T$ distribution has been studied by many authors in the real
case, see \citet{di:67}, \citet{bt:79}, \citet{p:82}, \citet{kn:04} and \citet{jdggj:09c},
among many others. The matricvariate $T$ distribution appears in the frequentist approach to
normal regression as the distribution of the Studentised error, see \citet{jdggj:06} and
\citet{kn:04}. In Bayesian conjugate-prior and diffuse-prior analysis for the same sampling
models, it appears as the marginal prior or posterior distribution of the unknown coefficients
matrix, and also as the predictive distribution of a future data array, see \citet{di:67},
\citet{bt:79}, \citet{p:82}, \citet{jdgrr:03}, \citet{fl:99} and \citet{kn:04}. It has been
applied in microeconomic modeling to describe the operation of a market for a particular
economic commodity and in macroeconomic modeling to describe the interrelations between a large
number of macroeconomic variables, as an application of the linear simultaneous equation model,
see \citet{kn:04}. In the complex case, the matricvariate $T$ distribution has been applied in
Bayesian estimation of a multivariate regression model, see \citet{kn:04}. No less important is
the role of the $T$ distribution, because if the matrix $\mathbf{T}$ has a $T$ distribution,
then the matrix $\mathbf{TT}^{*}$ (or $\mathbf{T}^{*}\mathbf{T}$) is distributed as beta type
II; and the distribution of the latter, in particular, plays a fundamental role in the MANOVA
model, see \citet{sk:79}, \citet{p:82}, and \citet{m:82}.

In this work, the nonsingular central matricvariate $T$ and the beta type II distributions and
some of their basic properties are studied, see Section \ref{sec2}. The matrix multivariate $T$
distribution and its corresponding beta type II distribution is studied in Section \ref{sec3}.
Finally, the joint densities of the singular values are derived in Section \ref{sec4}. We
emphasize that all these results are derived for real normed division algebras. Some concepts
and the notation of abstract algebra and Jacobians are summarised as an Appendix.

\section{Matricvariate $T$ distribution}\label{sec2}

We begin this section by distinguishing between \emph{matricvariate} and \emph{matrix
multivariate} (or \emph{matrix variate}) distributions. We say that the random matrix
$\mathbf{X}$ has a \emph{matricvariate distribution} if the kernel $g(\cdot)$ of its density is
written as a function solely in terms of the determinant operator $g(|\mathbf{X}|)$. In any
other case it is said that $\mathbf{X}$ has a \emph{matrix multivariate} distribution. The term
matricvariate distribution was introducing by \citet{di:67}, but the expression matrix-variate
distribution (or matrix variate distribution) was later used to describe any distribution of a
random matrix, see \citet{gv:93}, \citet{gn:00}, and references therein. Alternatively, the
term matrix multivariate (instead of matrix variate) has been used by \citet{gm:93}, and this
is the approach adopted in our paper.

\begin{thm}\label{teo1}
Let $\mathbf{T}\in {\mathcal L}_{m,n}^{\beta}$ defined as
$$
  \mathbf{T} = \mathbf{L}^{-1}\mathbf{Y} + \boldsymbol{\mu}
$$
where $\mathbf{L}$ is any square root of $\mathbf{V}$ such that $\mathbf{LL}^{*} = \mathbf{V
}\sim \mathcal{W}_{m}^{\beta}(\nu, \mathbf{\Xi})$, $\mathbf{\Xi} \in \mathfrak{P}_{m}^{\beta}$
and $\nu > \beta(m-1)$; independent of $\mathbf{Y} \sim \mathcal{N}_{m \times
n}^{\beta}(\mathbf{0}, \mathbf{I}_{m}\otimes \mathbf{\Sigma})$, $\mathbf{\Sigma} \in
\mathfrak{P}_{n}^{\beta}$. Then the density of $\mathbf{T}$ is
\begin{equation}\label{T}
    \frac{\Gamma_{m}^{\beta}[\beta(n+\nu)/2]}{\pi^{mn\beta/2}\Gamma_{m}^{\beta}[\beta \nu/2]
    |\mathbf{\Xi}|^{\beta \nu/2}|\mathbf{\Sigma}|^{\beta m/2}}
    |\mathbf{\Xi}^{-1} + (\mathbf{T} - \boldsymbol{\mu})\mathbf{\Sigma}^{-1}(\mathbf{T} -
    \boldsymbol{\mu})^{*}|^{-\beta(n+\nu)/2},
\end{equation}
which is termed the \emph{matricvariate $T$ distribution}\footnote{In the literature, it is
customary to use the expressions real matricvariate $T$ distribution, complex matricvariate $T$
distribution, quaternion matricvariate $T$ distribution and octonion matricvariate $T$
distribution distribution; here, however, we use simply matricvariate $T$ distribution as the
generic term.} and is denoted as
$$
  \mathbf{T} \sim \mathcal{T}_{m \times n}^{\beta}(\nu, \boldsymbol{\mu},
  \mathbf{\Xi}, \mathbf{\Sigma}).
$$
\end{thm}
\begin{proof}
From \citet{k:84} and \citet{jdggj:09a,jdggj:10b}, the joint density of $\mathbf{V}$ and
$\mathbf{Y}$ is
$$
  \propto |\mathbf{V}|^{\beta(\nu-m+1)/2-1}\etr\{-\beta(\mathbf{\Xi}^{-1}
  \mathbf{V} + \mathbf{\Sigma}^{-1}\mathbf{Y}^{*}\mathbf{Y})/2\},
$$
where the constant of proportionality is
$$
  c = \frac{1}{(2\beta^{-1})^{\beta m\nu/2} \Gamma_{m}^{\beta}[\beta \nu/2] |\mathbf{\Xi}|^{\beta \nu/2}}
      \ \cdot \ \frac{1}{(2\pi\beta^{-1})^{\beta mn/2} |\mathbf{\Sigma}|^{\beta m/2}}.
$$
Making the change of variable  $\mathbf{Y} = \mathbf{L}\mathbf{T}$, where $\mathbf{V} =
\mathbf{LL}^{*}$, then by (\ref{lt})
$$
  (d\mathbf{V})(d\mathbf{Y}) = |\mathbf{LL}^{*}|^{\beta n/2}(d\mathbf{V})(d\mathbf{T})
    = |\mathbf{V}|^{\beta n/2}(d\mathbf{V})(d\mathbf{T}).
$$
Thus, the joint density of $\mathbf{V}$ and $\mathbf{T}$ is
$$
  \propto |\mathbf{V}|^{\beta(\nu+n-m+1)/2-1}
  \etr\{-\beta(\mathbf{\Xi}^{-1} + (\mathbf{T} - \boldsymbol{\nu})\mathbf{\Sigma}^{-1}(\mathbf{T} -
    \boldsymbol{\nu})^{*})/2\}.
$$
Finally, integrating over $\mathbf{V} \in \mathfrak{P}_{m}^{\beta}$, we have
$$
  (2\beta^{-1})^{\beta m(n+\nu)/2} \Gamma[\beta(n+\nu)/2]|\mathbf{\Xi}^{-1} +
  (\mathbf{T} - \boldsymbol{\nu})\mathbf{\Sigma}^{-1}(\mathbf{T} -
  \boldsymbol{\nu})^{*}|^{-\beta(n+\nu)/2},
$$
from which the desired result is obtained. \qed
\end{proof}
Now, observe that by \citet{di:67}
$$
  \frac{\Gamma_{m}^{\beta}[\beta(n+\nu)/2]}{\pi^{mn\beta/2}\Gamma_{m}^{\beta}[\beta \nu/2]} =
  \frac{\Gamma_{n}^{\beta}[\beta(n+\nu)/2]}{\pi^{mn\beta/2}\Gamma_{n}^{\beta}[\beta(n+\nu-m)/2]}
$$
and
$$
  |\mathbf{\Xi}^{-1} + (\mathbf{T} - \boldsymbol{\nu})\mathbf{\Sigma}^{-1}(\mathbf{T} -
  \boldsymbol{\nu})^{*}| = |\mathbf{\Xi}|^{-1}|\mathbf{\Sigma}|^{-1}|\mathbf{\Sigma} +
  (\mathbf{T} - \boldsymbol{\nu})^{*}\mathbf{\Xi}(\mathbf{T} -
  \boldsymbol{\nu})|,
$$
from which, alternatively, the density (\ref{T}) can be expressed as
\begin{equation}\label{T2}
    \frac{\Gamma_{n}^{\beta}[\beta(n+\nu)/2] |\mathbf{\Xi}|^{\beta n/2}|\mathbf{\Sigma}|^{\beta(n+\nu-m)/2}}
    {\pi^{mn\beta/2}\Gamma_{n}^{\beta}[\beta(n+\nu-m)/2]} |\mathbf{\Sigma} + (\mathbf{T} -
    \boldsymbol{\nu})^{*}\mathbf{\Xi}(\mathbf{T} - \boldsymbol{\nu})|^{-\beta(n+\nu)/2},
\end{equation}
\begin{cor}\label{cor3}
Let $\mathbf{T}\in {\mathcal L}_{m,n}^{\beta}$ defined as
$$
  \mathbf{T} = \mathbf{X}\mathbf{L}_{1}^{-1} + \boldsymbol{\mu}
$$
where $\mathbf{L}_{1}$ is any square root of $\mathbf{U}$ such that
$\mathbf{L}_{1}\mathbf{L}_{1}^{*} = \mathbf{U} \sim \mathcal{W}_{n}^{\beta}(\nu+n-m,
\mathbf{\Sigma}^{-1})$, $\mathbf{\Sigma} \in \mathfrak{P}_{n}^{\beta}$, independent of
$\mathbf{X} \sim \mathcal{N}_{m \times n}^{\beta}(\mathbf{0},\mathbf{\Xi}^{-1} \otimes
\mathbf{I}_{n})$, with $\mathbf{\Xi} \in \mathfrak{P}_{m}^{\beta}$. Then, $\mathbf{T} \sim
\mathcal{T}_{m \times n}^{\beta}(\nu, \boldsymbol{\mu}, \mathbf{\Xi}, \mathbf{\Sigma})$.
\end{cor}
\begin{proof}
The proof is a verbatim copy of the proof of Theorem \ref{teo1}. \qed
\end{proof}

Now, assume that $\mathbf{T} \sim \mathcal{T}_{m \times n}^{\beta}(\nu, \boldsymbol{0},
\mathbf{I}_{m},\mathbf{I}_{n})$ with $n \geq m$ and let $\mathbf{F} \in
\mathfrak{P}_{m}^{\beta}$ defined as $\mathbf{F} = \mathbf{TT}^{*}$ then, under the conditions
of Theorem \ref{teo1} and Corollary \ref{cor3}, we have
\begin{eqnarray*}
  \mathbf{F} &=& \mathbf{L}^{-1}\mathbf{YY}^{*}(\mathbf{L}^{-1})^{*} = \mathbf{L}^{-1}\mathbf{S}(\mathbf{L}^{-1})^{*} \\
   &=& \mathbf{X}\mathbf{U}^{-1}\mathbf{X}^{*},
\end{eqnarray*}
with $\mathbf{S}=\mathbf{YY}^{*} \sim \mathcal{W}_{m}^{\beta}(n, \mathbf{I}_{m})$, $n
> \beta(m-1)$. Thus:
\begin{thm}\label{teo2}
The density of $\mathbf{F}$ is
\begin{equation}\label{FII1}
    \frac{1}{\mathcal{B}_{m}^{\beta}[\beta \nu/2, \beta n/2]}|\mathbf{F}|^{\beta(n-m+1)/2-1}
    |\mathbf{I}_{m}+\mathbf{F}|^{-\beta(n+\nu)/2},
\end{equation}
where $\mathcal{B}_{m}^{\beta}[\cdot, \cdot]$ is given by (\ref{beta}) and $\mathbf{F}$ is said
to have a \emph{matricvariate beta type II distribution}.
\end{thm}
\begin{proof}
The proof follows from (\ref{T}) by applying (\ref{vol}) and (\ref{w}). \qed
\end{proof}

In addition, assume that $n < m$ and let  $\widetilde{\mathbf{F}} \in \mathfrak{P}_{n}^{\beta}$
defined as $\widetilde{\mathbf{F}} = \mathbf{T}^{*}\mathbf{T}$ then, under the conditions of
Theorem \ref{teo1} and Corollary \ref{cor3} we have
\begin{eqnarray*}
  \widetilde{\mathbf{F}} &=& \mathbf{X}^{*}\mathbf{U}^{-1}\mathbf{X}\\
   &=& \mathbf{L}_{1}^{-1}\mathbf{X}^{*}\mathbf{X}(\mathbf{L}_{1}^{-1})^{*} =
   \mathbf{L}_{1}^{-1}\mathbf{S}_{1}(\mathbf{L}_{1}^{-1})^{*}
\end{eqnarray*}
with $\mathbf{S}_{1} = \mathbf{X}^{*}\mathbf{X} \sim \mathcal{W}_{n}^{\beta}(m,
\mathbf{I}_{n})$, $m
> \beta(n-1)$, Thus:
\begin{thm}\label{teo3}
$\widetilde{\mathbf{F}}$ has the density
\begin{equation}\label{FII2}
    \frac{1}{\mathcal{B}_{n}^{\beta}[\beta (\nu+n-m)/2, \beta m/2]}|\widetilde{\mathbf{F}}|^{\beta(m-n+1)/2-1}
    |\mathbf{I}_{n}+\widetilde{\mathbf{F}}|^{-\beta(n+\nu)/2-1}.
\end{equation}
Furthermore, we say that $\widetilde{\mathbf{F}}$ has a \emph{matricvariate beta type II
distribution}.
\end{thm}
\begin{proof}
The proof is the same as that given in Theorem \ref{teo2}. Alternatively, observe that density
(\ref{FII2}) can be obtained from density (\ref{FII1}) making the following substitutions, see
\citet[Eq. (7), p. 455]{m:82} and \citet[p. 96]{sk:79},
\begin{equation}\label{s}
    m \rightarrow n, \quad n \rightarrow m, \quad \nu \rightarrow \nu+n-m. \qquad \mbox{\qed}
\end{equation}
\end{proof}

Finally, assume that $\mathbf{M} \in \mathcal{L}_{m \times m}^{\beta}$ is any square root of
constant matrix $\mathbf{\Delta} = \mathbf{MM}^{*} \in \mathfrak{P}_{m}^{\beta}$. Also, define
$\mathbf{Z} = \mathbf{M}^{*}\mathbf{F}\mathbf{M}$, therefore:
\begin{cor} The density of $\mathbf{Z}$ is
$$
  \frac{|\mathbf{\Delta}|^{\beta\nu/2}}{\mathcal{B}_{m}^{\beta}[\beta \nu/2, \beta n/2]}|\mathbf{Z}|^{\beta(n-m+1)/2-1}
    |\mathbf{\Delta}+\mathbf{Z}|^{-\beta(n+\nu)/2}.
$$
$\mathbf{Z}$ is said to have a \emph{nonstandardised matricvariate beta type II distribution}.
\end{cor}
\begin{proof}
The proof follows from (\ref{FII1}) by applying (\ref{hlt}). \qed
\end{proof}

Densities (\ref{FII1}) and (\ref{FII2}) have been studied by several authors in the real case,
see \citet{k:70} and \citet{sk:79}, \citet{gn:00}, among many others; and by \citet{j:64},
\citet{m:82}, \citet{jdggj:09b} and \citet{jdggj:08} and \citet{jdggj:10a}, in the noncentral,
doubly noncentral, singular and nonsingular and complex cases, among many other authors.

\section{Matrix multivariate $T$ distribution}\label{sec3}

\begin{thm}\label{teo4}
Let $\mathbf{T}_{1} = S^{-1/2}\mathbf{Y}+ \boldsymbol{\mu} \in \mathcal{L}_{m,n}^{\beta}$ where
$(S^{1/2})^{2} = S \sim \mathbf{\Gamma}^{\beta}(\nu, \rho)$ (that is, $S$ has a gamma
distribution with parameters $\nu$ and $\rho$), $\rho
> 0$, independent of $\mathbf{Y} \sim \mathcal{N}_{m \times n}^{\beta}(\mathbf{0},
\mathbf{I}_{m} \otimes \mathbf{\Sigma})$, $\mathbf{\Sigma} \in \mathfrak{P}_{n}^{\beta}$. Then
the density of $\mathbf{T}_{1}$ is
\begin{equation}\label{mmt}
    \frac{\Gamma^{\beta}_{1}[\beta(\nu+mn)/2] \rho^{\beta mn/2}}{\pi^{\beta mn/2}
    \Gamma^{\beta}_{1}[\beta \nu/2]|\mathbf{\Sigma}|^{\beta m/2}}
    \left[1+\rho\tr \mathbf{\Sigma}^{-1}(\mathbf{T}_{1}- \boldsymbol{\mu})^{*}
    (\mathbf{T}_{1}- \boldsymbol{\mu}) \right]^{-\beta(\nu +mn)/2},
\end{equation}
which is termed the \emph{matrix multivariate $T$ distribution} and is denoted as
$\mathbf{T}_{1} \sim \mathcal{MT}_{m \times n}^{\beta}(\nu,\boldsymbol{\mu}, \mathbf{I}_{m},
\mathbf{\Sigma})$.
\end{thm}
\begin{proof}
The joint density of $S$ and $\mathbf{Y}$ is
$$
  \propto s^{\beta \nu/2 -1} \etr\{-\beta (s/\rho + \tr \mathbf{YY}^{*})/2\},
$$
where the constant of proportionality is
$$
  c = \frac{1}{(2\beta^{-1})^{\beta \nu/2} \Gamma_{1}^{\beta}[\beta \nu/2] \rho^{\beta \nu/2}}
      \ \cdot \ \frac{1}{(2\pi\beta^{-1})^{\beta mn/2} |\mathbf{\Sigma}|^{\beta m/2}}.
$$
Noting that by (\ref{lt})
$$
  (ds)(d\mathbf{Y})= s^{\beta mn/2}(ds)(d\mathbf{T}_{1}),
$$
the desired result is obtained analogously to the proof of Theorem \ref{teo1}. \qed
\end{proof}
\begin{cor}
Assume that $\mathbf{T}_{1} \sim \mathcal{MT}_{m \times n}^{\beta}(\nu,\boldsymbol{\mu},
\mathbf{I}_{m},\mathbf{I}_{n})$, and let $\mathbf{M}$ and $\mathbf{N}$ be any square root of
the constant matrices $\mathbf{\Delta} = \mathbf{M}\mathbf{M}^{*}\in \mathfrak{P}_{m}^{\beta}$
and $\mathbf{\Lambda}= \mathbf{N}\mathbf{N}^{*}\in \mathfrak{P}_{n}^{\beta}$, respectively.
Also let $(\mathbf{M}^{*})^{-1}\mathbf{T}_{1}\mathbf{N}^{-1} + \boldsymbol{\mu} =
\mathbf{Q}_{1} \in \mathcal{L}_{m,n}^{\beta}$, where $\boldsymbol{\mu} \in
\mathcal{L}_{m,n}^{\beta}$ is constant. Then,
$$
    \frac{\Gamma^{\beta}_{1}[\beta(\nu+mn)/2]}{\pi^{\beta mn/2}\Gamma^{\beta}_{1}[\beta \nu/2]}
    |\mathbf{\Delta}|^{\beta n/2} |\mathbf{\Lambda}|^{\beta
    m/2}\left[1-\tr \mathbf{\Delta}(\mathbf{Q}_{1} - \boldsymbol{\mu})\mathbf{\Lambda}
    (\mathbf{Q}_{1} - \boldsymbol{\mu})^{*}\right]^{-\beta(\nu +mn)/2},
$$
Hence, we write $\mathbf{Q}_{1} \sim \mathcal{MT}_{m \times n}^{\beta}(\nu, \boldsymbol{\mu},
\mathbf{\Delta}, \mathbf{\Lambda})$.
\end{cor}
\begin{proof}
The proof follows observing that, by (\ref{lt})
$$
  (d\mathbf{T}_{1}) = |\mathbf{M}\mathbf{M}^{*}|^{\beta n/2} |\mathbf{N}\mathbf{N}^{*}|^{\beta
    m/2}(d\mathbf{Q}_{1}) = |\mathbf{\Delta}|^{\beta n/2} |\mathbf{\Lambda}|^{\beta
    m/2}(d\mathbf{Q}_{1}). \qquad \mbox{\qed}
$$
\end{proof}
Now, assuming that $\mathbf{T}_{1} \sim \mathcal{MT}_{m \times n}^{\beta}(\nu, \boldsymbol{0},
\mathbf{I}_{m},\mathbf{I}_{n}),$ with $n \geq m$ and defining $\mathbf{F}_{1} =
\mathbf{T}_{1}\mathbf{T}_{1}^{*} \in \mathfrak{P}_{m}^{\beta}$, then, under the conditions of
Theorem \ref{teo4} we have that
$$
  \mathbf{F}_{1} = S^{-1}\mathbf{YY}^{*} = S^{-1}\mathbf{W}
$$
where $\mathbf{W}=\mathbf{YY}^{*} \sim \mathcal{W}_{m}^{\beta}(n, \mathbf{I}_{m})$, $n
> \beta(m-1)$.
\begin{thm}\label{teo5}
The density of $\mathbf{F}_{1}$ is
\begin{equation}\label{MFII1}
    \frac{\Gamma^{\beta}_{1}[\beta(\nu+mn)/2]}{\Gamma^{\beta}_{1}[\beta \nu/2] \Gamma_{m}^{\beta}[\beta n/2]}
    |\mathbf{F}_{1}|^{\beta(n-m+1)/2-1}(1+\tr \mathbf{F}_{1})^{-\beta(\nu+mn)/2},
\end{equation}
$\mathbf{F}_{1}$ is said to have a \emph{matrix multivariate beta type II distribution}.
\end{thm}
\begin{proof}
The proof follows from (\ref{mmt}) by applying (\ref{vol}) and (\ref{w}). \qed
\end{proof}
Similarly, if $n<m$ and $\widetilde{\mathbf{F}}_{1} = \mathbf{T}_{1}^{*}\mathbf{T}_{1} \in
\mathfrak{P}_{n}^{\beta}$.
\begin{thm}\label{teo6}
$\widetilde{\mathbf{F}}_{1}$ has the density
\begin{equation}\label{MFII2}
    \frac{\Gamma^{\beta}_{1}[\beta(\nu+mn)/2]}{\Gamma^{\beta}_{1}[\beta \nu/2] \Gamma_{n}^{\beta}[\beta m/2]}
    |\widetilde{\mathbf{F}}_{1}|^{\beta(m-n+1)/2-1}(1+\tr \widetilde{\mathbf{F}}_{1})^{-\beta(\nu+mn)/2}.
\end{equation}
Thus, $\widetilde{\mathbf{F}}_{1}$ is said to have a \emph{matrix multivariate distribution
type II distribution}.
\end{thm}
\begin{proof}
The proof is the same as that given in Theorem \ref{teo5}. Alternatively, the density
(\ref{MFII2}) can be obtained from density (\ref{MFII1}) by making the following substitutions,
\begin{equation}\label{s2}
    m \rightarrow n, \quad n \rightarrow m. \qquad \qquad \mbox{\qed}
\end{equation}
\end{proof}
As in the matricvariate beta type II distributions, assume that $\mathbf{A} \in \mathcal{L}_{m
\times m}^{\beta}$ is any square root of constant matrix $\mathbf{\Pi} = \mathbf{AA}^{*} \in
\mathfrak{P}_{m}^{\beta}$. Also, define $\mathbf{Z} = \mathbf{A}^{*}\mathbf{F}\mathbf{A}$,
therefore:
\begin{cor} The density of $\mathbf{Z}$ is
$$
  \frac{\Gamma^{\beta}_{1}[\beta(\nu+mn)/2]|\mathbf{\Pi}|^{\beta n/2}}{\Gamma^{\beta}_{1}[\beta \nu/2]
  \Gamma_{n}^{\beta}[\beta m/2]}|\mathbf{Z}|^{\beta(n-m+1)/2-1}
    (1+\mathbf{\Pi Z})^{-\beta(n+\nu)/2}.
$$
$\mathbf{Z}$ is said to have a \emph{nonstandardised matrix multivariate beta type II
distribution}.
\end{cor}
\begin{proof}
The proof follows from (\ref{MFII1}) by applying (\ref{hlt}). \qed
\end{proof}
In the real and singular case, the matricvariate and matrix multivariate $T$ distributions have
been studied by \citet{jdggj:09c}.

\section{Singular value densities}\label{sec4}
In this section, the joint densities of the singular values of matrices $\mathbf{T}$,
$\widetilde{\mathbf{T}}$, $\mathbf{T}_{1}$ and $\widetilde{\mathbf{T}}_{1}$ are derived. In
addition, and as a direct consequence, the joint densities of the eigenvalues of $\mathbf{F}$,
$\widetilde{\mathbf{F}}$, $\mathbf{F}_{1}$ and $\widetilde{\mathbf{F}}_{1}$ are obtained for
real normed division algebras.

\begin{thm}\label{teosv}
Let $\delta_{1}, \dots, \delta_{m}$ be the singular values of $\mathbf{T} \sim \mathcal{T}_{m
\times n}^{\beta}(\nu,\mathbf{0}, \mathbf{I}_{m}, \mathbf{I}_{n})$, $\delta_{1}> \cdots >
\delta_{m} > 0$. Then its joint density is
\begin{equation}\label{svT}
    \frac{2^{m} \ \pi^{\beta m^{2}+\tau}}{\Gamma_{m}^{\beta}[\beta m/2] \mathcal{B}_{m}^{\beta}[\beta \nu/2, \beta n/2]}
    \prod_{i=1}^{m} \delta_{i}^{\beta(n-m+1)-1}(1+\delta_{i}^{2})^{-\beta(\nu+n)/2}
    \prod_{i<j}^{m}(\delta_{i}^{2} - \delta_{j}^{2})^{\beta}
\end{equation}
where $\tau$ is defined in Lemma \ref{lemsvd}.
\end{thm}
\begin{proof}
This follows immediately from (\ref{T}), first using (\ref{svd}) and then applying (\ref{vol}).
\qed
\end{proof}
The joint density of the singular values of $\widetilde{\mathbf{T}}$ is obtained from
(\ref{svT}) after making the substitutions (\ref{s}).

\begin{thm}
Assume that $\mathbf{T}_{1} \sim \mathcal{MT}_{m \times n}^{\beta}(\nu,\mathbf{0},
\mathbf{I}_{m}, \mathbf{I}_{n})$ and let $\alpha_{1}, \dots, \alpha_{m}$, $\alpha_{1}> \cdots
>, \alpha_{m} > 0$, be its singular values. Then its joint density is
\begin{small}
\begin{equation}\label{svMT}
  \frac{2^{m} \pi^{\beta m^{2}/2 + \tau}\Gamma_{1}^{\beta}[\beta(\nu +
  mn)/2]}{\Gamma_{1}^{\beta}[\beta \nu/2]
  \Gamma_{m}^{\beta}[\beta m/2]\Gamma_{m}^{\beta}[\beta n/2]}
  \prod_{i=1}^{m} \alpha_{i}^{\beta(n-m+1)-1}
  \left(1+\sum_{i=1}^{m}\alpha_{i}^{2}\right)^{-\beta(\nu+mn)/2}
  \prod_{i<j}^{m}(\alpha_{i}^{2} - \alpha_{j}^{2})^{\beta}
\end{equation}
\end{small}
\end{thm}
\begin{proof}
The proof is identical to that given for Theorem \ref{teosv}. \qed
\end{proof}
Analogously, the joint density of the singular values of $\widetilde{\mathbf{T}}_{1}$ is
obtained from (\ref{svMT}), making the substitutions (\ref{s2}).

Finally, observe that $\delta_{i} = \sqrt{\eig_{i}(\mathbf{TT}^{*})}$  and $\alpha_{i} =
\sqrt{\eig_{i}(\mathbf{T}_{1}\mathbf{T}_{1}^{*})}$, where $\eig_{i}(\mathbf{A})$, $i = 1,
\dots, m$, denotes the $i$-th eigenvalue of $\mathbf{A}$. Let $\lambda_{i} =
\eig_{i}(\mathbf{TT}^{*}) = \eig_{i}(\mathbf{F})$ and $\gamma_{i} =
\eig_{i}(\mathbf{T}_{1}\mathbf{T}_{1}^{*}) = \eig_{i}(\mathbf{F}_{1})$, observing that, for
example, $\delta_{i} = \sqrt{\lambda_{i}}$. Then
$$
  \bigwedge_{i=1}^{m} d\delta_{i} = \bigwedge_{i=1}^{m} 2^{-m} \prod_{i=1}^{m}
  \lambda_{i}^{-1/2}d\lambda_{i},
$$
the corresponding joint density of $\lambda_{1}, \dots, \lambda_{m}$, $\lambda_{1} > \cdots
> \lambda_{m} > 0$ is obtained from (\ref{svT}) as
$$
    \frac{\pi^{\beta m^{2}+\tau}}{\Gamma_{m}^{\beta}[\beta m/2] \mathcal{B}_{m}^{\beta}[\beta \nu/2, \beta n/2]}
    \prod_{i=1}^{m} \lambda_{i}^{\beta(n-m+1)/2-1}(1+\lambda_{i})^{-\beta(\nu+n)/2}
    \prod_{i<j}^{m}(\lambda_{i} - \lambda_{j})^{\beta}.
$$
Analogously, the joint density of $\gamma_{1}, \dots, \gamma_{m}$, $\gamma_{1} > \cdots >
\gamma_{m}> 0$, is obtained from (\ref{svMT}) as
\begin{small}
$$
  \frac{\pi^{\beta m^{2}/2 + \tau}\Gamma_{1}^{\beta}[\beta(\nu +
  mn)/2]}{\Gamma_{1}^{\beta}[\beta \nu/2]
  \Gamma_{m}^{\beta}[\beta m/2]\Gamma_{m}^{\beta}[\beta n/2]}
  \prod_{i=1}^{m} \gamma_{i}^{\beta(n-m+1)/2-1}
  \left(1+\sum_{i=1}^{m}\gamma_{i}\right)^{-\beta(\nu+mn)/2}
  \prod_{i<j}^{m}(\gamma_{i} - \gamma_{j})^{\beta}.
$$
\end{small}
\begin{rem}
Observe that $\mathbf{Y}\in \mathfrak{L}^{\beta}_{m,n}$ has a matrix multivariate elliptically
contoured distribution for real normed division algebras if its density, with respect to the
Lebesgue measure, is given by (see \citet{jdggj:09a}):
\begin{equation}\label{mve}
  \frac{C^{\beta}(m,n)}{|\mathbf{\Sigma}|^{\beta n/2}|\mathbf{\Theta}|^{\beta m/2}}
  h\left\{\tr\left[\mathbf{\Sigma}^{-1}(\mathbf{Y}-\boldsymbol{\mu})\mathbf{\Theta}^{-1}
  (\mathbf{Y}- \boldsymbol{\mu})^{*}\right]\right\},
\end{equation}
where  $\boldsymbol{\mu}\in \mathfrak{L}^{\beta}_{m,n}$, $ \mathbf{\Sigma}\in
\mathfrak{P}^{\beta}_{m}$,  $ \mathbf{\Theta}\in \mathfrak{P}^{\beta}_{m}$. The function $h:
\mathfrak{F} \rightarrow [0,\infty)$ is termed the generator function, and it is such that
$\int_{\mathfrak{P}^{\beta}_{1}} u^{\beta nm-1}h(u^2)du < \infty$ and
$$
  C^{\beta}(m,n) = \frac{\Gamma[\beta mn/2]}{2 \pi^{\beta mn/2}} \left\{
    \int_{\mathfrak{P}^{\beta}_{1}} u^{\beta nm-1}h(u^2)du\right \}
$$
Such a distribution is denoted by $\mathbf{Y}\sim \mathcal{E}^{\beta}_{n\times
m}(\boldsymbol{\mu},\mathbf{\Sigma}, \mathbf{\Theta}, h)$, for the real case, see \citet{fz:90}
and \citet{gv:93}; and \citet{mdm:06} for the complex case. Observe that this class of matrix
multivariate distributions includes normal, contaminated normal, Pearson type II and VII, Kotz,
Jensen-Logistic, power exponential and Bessel distributions, among others; these distributions
have tails that are more or less weighted, and/or present a greater or smaller degree of
kurtosis than the normal distribution.

Assume that $\mathbf{Y} = (\build{\mathbf{Y}_{1}}{}{m \times n}\vdots\build{\mathbf{Y}_{2}}{}{m
\times \nu}) \sim \mathcal{E}^{\beta}_{m \times n+\nu}(\boldsymbol{0},\mathbf{I}_{m},
\mathbf{I}_{n+\nu}, h)$, $n,\nu \geq m$; and define, $\mathbf{T} =
\mathbf{L}^{-1}\mathbf{Y}_{1}$, where $\mathbf{L}$ is any square root of $\mathbf{V} =
\mathbf{Y}_{2}\mathbf{Y}_{2}^{*}$ such that $\mathbf{LL}^{*} = \mathbf{V}$. Then $\mathbf{T}
\sim \mathcal{T}_{m \times n}^{\beta}(\nu,\mathbf{0}, \mathbf{I}_{m},\mathbf{I}_{n})$. From
(\ref{mve}) the density of $\mathbf{Y}$ is
$$
  C^{\beta}(m,n+\nu) h\left\{\tr\left(\mathbf{Y}_{1}\mathbf{Y}_{1}^{*} +
  \mathbf{Y}_{2}\mathbf{Y}_{2}^{*}\right)\right\}.
$$
Let $\mathbf{V} = \mathbf{Y}_{2}\mathbf{Y}_{2}^{*}$ then by (\ref{lemW}), $(d\mathbf{Y}_{2}) =
2^{-m}|\mathbf{V}|^{\beta(\nu -m+1)/2-1}(d\mathbf{V}) (\mathbf{H}_{1}d\mathbf{H}_{1}^{*})$.
Thus, the marginal density of $\mathbf{Y}_{1}$ and $\mathbf{V}$ is obtained by integrating over
$\mathbf{H}_{1} \in \mathcal{V}_{m,n}^{\beta}$ by using (\ref{vol}), hence
$$
  \frac{C^{\beta}(m,n+\nu) \pi^{\beta \nu m/2}}{\Gamma_{m}^{\beta}[\beta \nu/2]} |\mathbf{V}|^{\beta(\nu -m+1)/2-1}
  h\left\{\tr\left(\mathbf{Y}_{1}\mathbf{Y}_{1}^{*} +  \mathbf{V}\right)\right\}.
$$
Now, let $\mathbf{T} = \mathbf{L}^{-1}\mathbf{Y}_{1}$, where $\mathbf{L}\mathbf{L}^{*}
=\mathbf{V}$, then by (\ref{lemlt})
$$
  (d\mathbf{Y}_{1})(d\mathbf{V}) =  |\mathbf{V}|^{\beta n/2}(d\mathbf{T})(d\mathbf{V}).
$$
Therefore, the joint density of $\mathbf{T}$ and $\mathbf{V}$ is
$$
  \frac{C^{\beta}(m,n+\nu) \pi^{\beta \nu m/2}}{\Gamma_{m}^{\beta}[\beta \nu/2]}
  |\mathbf{V}|^{\beta(n+\nu -m+1)/2-1} h\left\{\tr(\mathbf{I}^{m} + \mathbf{TT}^{*})\mathbf{V}\right\}.
$$
The desired result follows by applying the next equally,
$$
  \int_{\mathbf{V} \in \mathfrak{P}_{m}^{\beta}} |\mathbf{V}|^{\beta(n+\nu -m+1)/2-1}
  h\left\{\tr \mathbf{A}\mathbf{V}\right\}(d\mathbf{V}) =
  \frac{\Gamma_{m}^{\beta}[\beta(n+\nu)2]|\mathbf{A}|^{-\beta(n+\nu)/2}}
  {\pi^{\beta m(n+\nu)/2} C^{\beta}(m, n+\nu)},
$$
see \citet{jdggj:09a} and \citet{jdggj:10b}.\qed

Observe that in this case, $\mathbf{Y}_{1}$ and $\mathbf{Y}_{2}$ (or $\mathbf{V} =
\mathbf{Y}_{2}\mathbf{Y}_{2}^{*}$) are stochastically dependent. Furthermore, note that only
when the particular matrix multivariate elliptical distribution is the matrix multivariate
normal distribution, are $\mathbf{Y}_{1}$ and $\mathbf{Y}_{2}$ (or $\mathbf{V} =
\mathbf{Y}_{2}\mathbf{Y}_{2}^{*}$) independent. Therefore, we can say that the matricvariate
$T$ distribution is invariant under the family of matrix multivariate elliptical distributions
for real normed division algebras, and furthermore, its density is the same as when normality
is assumed. Analogously, it can be proved that the matrix multivariate $T$, matricvariate and
matrix multivariate beta type II distributions are invariant under the family of matrix
multivariate elliptical distributions for real normed division algebras. Furthermore, this
invariance prevails under other classes of elliptical models for real normed division algebras,
see \citet{fz:90}, \citet{gv:93} and \citet{jdggj:09a}.
\end{rem}

\section*{Conclusions}

Any topic in statistical literature, is usually first studied in the real case, later in the
complex case, later for the quaternion case and exceptionally for the octonion case. From the
results presented in this paper, the real, complex, quaternion and octonion cases are obtained
by simply replacing $\beta$ with $1,2,4$ or $8$, respectively. Furthermore, as observed by
\citet{k:84}, these results can be extended to hypercomplex cases, that is, for biquaternion
and bioctonion algebras. Then, from the results presented here, the hypercomplex cases are
obtained by replacing $\beta$ with $2\beta$.

\section*{Acknowledgements}
This research work was partially supported  by CONACYT-M\'exico, Research Grant No. \ 81512 and
IDI-Spain, Grants No. FQM2006-2271 and MTM2008-05785. This paper was written during J. A.
D\'{\i}az-Garc\'{\i}a's stay as a visiting professor at the Department of Statistics and O. R.
of the University of Granada, Spain.

\bibliographystyle{plain}

\begin{thebibliography}{}

\bibitem[Baez(2002)]{b:02}
    J. C. Baez,
    The octonions,
    Bull. Amer. Math. Soc.
    39 (2002) 145--205.

\bibitem[Bhavsar (2000)]{bh:00}
    C. D. Bhavsar,
    Asymptotic distributions of likelihood ratio criteria for two testing problems.
    Kybernetes
    29(4)(2000) 510--517.

\bibitem[Box and Tiao(1972)]{bt:79}
    G. E. Box, and G. C. Tiao,
    Bayesian Inference in Statistical Analysis, A
    ddison-Wesley Publishing Company, Reading, 1972.

\bibitem[D\'{\i}az-Garc\'{\i}a and  Guti\'errez-J\'aimez(2006)]{jdggj:06}
    J. A. D\'{\i}az-Garc\'{\i}a, R. Guti\'errez-J\'aimez,
    The distribution of the residual from a general elliptical multivariate
    linear regression model,
    J. Multivariate Anal. 97(2006) 1829-1841.

\bibitem[D\'{\i}az-Garc\'{\i}a and  Guti\'errez-J\'aimez(2008)]{jdggj:08}
    J. A. D\'{\i}az-Garc\'{\i}a, R. Guti\'errez-J\'aimez,
    Singular matrix variate beta distribution,
    J. Multivariate Anal. 99(2008) 637-648.

\bibitem[D\'{\i}az-Garc\'{\i}a and  Guti\'errez-J\'aimez(2009a)]{jdggj:09a}
    J. A. D\'{\i}az-Garc\'{\i}a, R. Guti\'errez-J\'aimez,
    Random matrix theory and multivariate statistics,
    \verb"http://arxiv.org/abs/0907.1064", (2009). Also submited.

\bibitem[D\'{\i}az-Garc\'{\i}a and  Guti\'errez-J\'aimez(2009b)]{jdggj:09b}
    J. A. D\'{\i}az-Garc\'{\i}a, R. Guti\'errez-J\'aimez,
    Doubly singular matrix variate beta type I and II and singular inverted
    matricvariate t distributions,
    J. Korean Statist. Soc. 38(3)(2009) 297-303.

\bibitem[D\'{\i}az-Garc\'{\i}a and  Guti\'errez-J\'aimez(2009c)]{jdggj:09c}
    J. A. D\'{\i}az-Garc\'{\i}a, R. Guti\'errez-J\'aimez,
    Singular matric and matrix variate t distributions,
    J. Statist. Plann. Inference 139(2009) 2382-2387.

\bibitem[D\'{\i}az-Garc\'{\i}a and  Guti\'errez-J\'aimez(2010a)]{jdggj:10a}
    J. A. D\'{\i}az-Garc\'{\i}a, R. Guti\'errez-J\'aimez,
    Doubly noncentral singular matrix variate beta distributions,
    J. Statist. Theory \& Practice 4(3)(2010) 421-431.

\bibitem[D\'{\i}az-Garc\'{\i}a and  Guti\'errez-J\'aimez(2010b)]{jdggj:10b}
    J. A. D\'{\i}az-Garc\'{\i}a, R. Guti\'errez-J\'aimez,
    On Wishart distribution. \verb"http://arxiv.org/abs/1010.1799", (2010). Also submited.

\bibitem[D\'{\i}az-Garc\'{\i}a and  Ramos-Quiroga(2003)]{jdgrr:03}
    J. A. D\'{\i}az-Garc\'{\i}a, R. Ramos-Quiroga,
    Generalised natural conjugate prior densities: Singular multivariate linear model,
    Int. Math. J.
    3 (12)(2003) 1279-1287.

\bibitem[Dickey(1967)]{di:67}
    J. M. Dickey,
    Matricvariate generalizations of  the multivariate  $t$- distribution and
    the inverted multivariate $t$-distribution,
    Ann. Math.Statist. 38 (1967) 511-518.

\bibitem[Dimitriu(2002)]{d:02}
    I. Dimitriu,
    Eigenvalue statistics for beta-ensembles.
    PhD thesis, Department of Mathematics,
    Massachusetts Institute of Technology, Cambridge, MA., 2002.

\bibitem[Edelman and Rao(2005)]{er:05}
    A. Edelman, R. R. Rao,
    Random matrix theory,
    Acta Numer.
    14 (2005) 233--297.

\bibitem[Fang and Li (1999)]{fl:99}
    Fang, K. T. and Li, R., 1999.
    Bayesian statistical inference on elliptical matrix distributions.
    J. Multivar. Anal.
    70, 66--85.

\bibitem[Fang and Zhang(1990)]{fz:90}
    K. T. Fang, Y. T. Zhang,
    Generalized Multivariate Analysis,
    Science Press, Beijing, Springer-Verlang, 1990.

\bibitem[Forrester(2009)]{f:05}
    P. J. Forrester,
    Log-gases and random matrices.
    To appear. Available in: \verb"http://www.ms.unimelb.edu.au/~matpjf/matpjf.html", 2009.

\bibitem[Goodall and Mardia(1993)]{gm:93}
    C. R. Goodall, and K. V. Mardia,
    Multivariate Aspects of Shape Theory,
    Ann.  Statist. 21 (1993), pp. 848--866.

\bibitem[Gross and Richards (1987)]{gr:87}
    K. I. Gross, and D. St. P. Richards,
    Special functions of matrix argument I: Algebraic induction zonal polynomials and hypergeometric
    functions.
    Trans. Amer. Math. Soc.
    301(2)(1987) 475--501.

\bibitem[Gupta and Nagar(2000)]{gn:00}
    A. K. Gupta, and D. K. Nagar,
    Matrix variate distributions,
    Chapman \& Hall/CR, New York, 2000.

\bibitem[Gupta, and  Varga(1993)]{gv:93}
    A. K. Gupta, T. Varga,
    Elliptically Contoured Models in Statistics,
    Kluwer Academic Publishers, Dordrecht, 1993.

\bibitem[Herz (1955)]{h:55}
   C. S. Herz,
   Bessel functions of matrix argument
   Ann. of Math.
   61(3)(1955) 474-523.

\bibitem[James(1964)]{j:64}
    A. T. James,
    Distribution of matrix variate and latent roots derived from normal samples,
    Ann. Math. Statist.
    35 (1964) 475--501.

\bibitem[Kabe(1984)]{k:84}
    D. G. Kabe,
    Classical statistical analysis based on a certain hypercomplex multivariate
    normal distribution,
    Metrika
    31(1964) 63--76.

\bibitem[Khatri(1965)]{k:65}
    C. G. Khatri,
    Classical statistical analysis based on a certain multivariate complex Gaussian
    distribution,
    Ann. Math. Statist.
    36(1) (1965) 98--114.

\bibitem[Khatri(1970)]{k:70}
    C. G. Khatri,
    A note on Mitra's paper ``A density free approach to the matrix variate beta distribution'',
    Sankhy\={a} A 32(1970) 311-318.

\bibitem[Kotz and Nadarajah(2004)]{kn:04}
    S. Kotz, and S. Nadarajah,
    Multivariate $t$ Distributions and Their Applications,
    Cambridge University Press, United Kingdom, 2004.

\bibitem[Li and Xue (2009)]{lx:09}
    F. Li, and Y. Xue,
    Zonal polynomials and hypergeometric functions of quaternion matrix argument,
    Comm. Statist. Theory Methods
    38(8)(2009) 1184-1206.

\bibitem[Mehta(1991)]{me:91}
    M. L. Mehta,
    Random matrices,
    Second edition Academic Press, Boston, 1991.

\bibitem[Micheas \emph{et al.}(2006)]{mdm:06}
    A. C. Micheas, D. K. Dey, and K. V. Mardia,
    Complex elliptical distribution with application to shape theory,
    J. Statist. Plann. Infer.
    136 (2006) 2961-2982.

\bibitem[Muirhead(1982)]{m:82}
    R. J.Muirhead,
    Aspects of Multivariate Statistical Theory,
    John Wiley \& Sons, New York, 1982.

\bibitem[Press(1982)]{p:82}
    S. J. Press,
    Applied Multivariate Analysis: Using Bayesian and Frequentist Methods of Inference,
    Second Edition, Robert E. Krieger
    Publishing Company, Malabar, Florida, 1982.

\bibitem[Ratnarajah \emph{et al.}(2005)]{rva:05a}
   T. Ratnarajah, R. Villancourt, and A. Alvo,
   Complex random matrices and Rician channel capacity.
   Probl. Inf. Transm.
   41(1)(2005), 1--22.

\bibitem[Srivastava \& Khatri(1979)]{sk:79}
    S. M. Srivastava, C. G. Khatri,  An
    introduction to multivariate statistics,
    North Holland, New York, 1979.

\end{thebibliography}

\appendix
\renewcommand{\theequation}{A-\arabic{equation}}
\setcounter{equation}{0}

\section*{Appendix}

A detailed discussion of real normed division algebras may be found in \citet{b:02} and
\citet{gr:87}. For convenience, we shall introduce some notation, although in general we adhere
to standard notation forms.

For our purposes, a \textbf{vector space} is always a finite-dimensional module over the field
of real numbers. An \textbf{algebra} $\mathfrak{F}$ is a vector space that is equipped with a
bilinear map $m: \mathfrak{F} \times \mathfrak{F} \rightarrow \mathfrak{F}$ termed
\emph{multiplication} and a nonzero element $1 \in \mathfrak{F}$ termed the \emph{unit} such
that $m(1,a) = m(a,1) = 1$. As usual, we abbreviate $m(a,b) = ab$ as $ab$. We do not assume
$\mathfrak{F}$ associative. Given an algebra, we freely think of real numbers as elements of
this algebra via the map $\omega \mapsto \omega 1$.

An algebra $\mathfrak{F}$ is a \textbf{division algebra} if given $a, b \in \mathfrak{F}$ with
$ab=0$, then either $a=0$ or $b=0$. Equivalently, $\mathfrak{F}$ is a division algebra if the
operation of left and right multiplications by any nonzero element is invertible. A
\textbf{normed division algebra} is an algebra $\mathfrak{F}$ that is also a normed vector
space with $||ab|| = ||a||||b||$. This implies that $\mathfrak{F}$ is a division algebra and
that $||1|| = 1$.

There are exactly four normed division algebras: real numbers ($\Re$), complex numbers
($\mathfrak{C}$), quaternions ($\mathfrak{H}$) and octonions ($\mathfrak{O}$), see
\citet{b:02}. We take into account that $\Re$, $\mathfrak{C}$, $\mathfrak{H}$ and
$\mathfrak{O}$ are the only normed division algebras; moreover, they are the only alternative
division algebras, and all division algebras have a real dimension of $1, 2, 4$ or $8$, which
is denoted by $\beta$, see \citet[Theorems 1, 2 and 3]{b:02}. In other branches of mathematics,
the parameters $\alpha = 2/\beta$ and $t = \beta/4$ are used, see \citet{er:05} and
\citet{k:84}, respectively.

Let ${\mathcal L}^{\beta}_{m,n}$ be the linear space of all $m \times n$ matrices of rank $m
\leq n$ over $\mathfrak{F}$ with $m$ distinct positive singular values, where $\mathfrak{F}$
denotes a \emph{real finite-dimensional normed division algebra}. Let $\mathfrak{F}^{m \times
n}$ be the set of all $m \times n$ matrices over $\mathfrak{F}$. The dimension of
$\mathfrak{F}^{m \times n}$ over $\Re$ is $\beta mn$. Let $\mathbf{A} \in \mathfrak{F}^{m
\times n}$, then $\mathbf{A}^{*} = \overline{\mathbf{A}}^{T}$ denotes the usual conjugate
transpose.

Table \ref{table1} sets out the equivalence between the same concepts in the four normed
division algebras.
\begin{table}[h]
  \centering
  \caption{Notation}\label{table1}
  \begin{footnotesize}
  \begin{tabular}{cccc|c}
    \hline
    Real & Complex & Quaternion & Octonion & \begin{tabular}{c}
                                               Generic \\
                                               notation \\
                                             \end{tabular}\\
    \hline
    Semi-orthogonal & Semi-unitary & Semi-symplectic & \begin{tabular}{c}
                                                         Semi-exceptional \\
                                                         type \\
                                                       \end{tabular}
      & $\mathcal{V}_{m,n}^{\beta}$ \\
    Orthogonal & Unitary & Symplectic & \begin{tabular}{c}
                                                         Exceptional \\
                                                         type \\
                                                       \end{tabular} & $\mathfrak{U}^{\beta}(m)$ \\
    Symmetric & Hermitian & \begin{tabular}{c}
                              Quaternion \\
                              hermitian \\
                            \end{tabular}
     & \begin{tabular}{c}
                              Octonion \\
                              hermitian \\
                            \end{tabular} & $\mathfrak{S}_{m}^{\beta}$ \\
    \hline
  \end{tabular}
  \end{footnotesize}
\end{table}

In addition, let $\mathfrak{P}_{m}^{\beta}$ be the \emph{cone of positive definite matrices}
$\mathbf{S} \in \mathfrak{F}^{m \times m}$; then $\mathfrak{P}_{m}^{\beta}$ is an open subset
of ${\mathfrak S}_{m}^{\beta}$.

Let $\mathfrak{D}_{m}^{\beta}$ be the \emph{diagonal subgroup} of $\mathcal{L}_{m,m}^{\beta}$
consisting of all $\mathbf{D} \in \mathfrak{F}^{m \times m}$, $\mathbf{D} = \diag(d_{1},
\dots,d_{m})$.

For any matrix $\mathbf{X} \in \mathfrak{F}^{n \times m}$, $d\mathbf{X}$ denotes the\emph{
matrix of differentials} $(dx_{ij})$. Finally, we define the \emph{measure} or volume element
$(d\mathbf{X})$ when $\mathbf{X} \in \mathfrak{F}^{m \times n}, \mathfrak{S}_{m}^{\beta}$,
$\mathfrak{D}_{m}^{\beta}$ or $\mathcal{V}_{m,n}^{\beta}$, see \citet{d:02}.

If $\mathbf{X} \in \mathfrak{F}^{m \times n}$ then $(d\mathbf{X})$ (the Lebesgue measure in
$\mathfrak{F}^{m \times n}$) denotes the exterior product of the $\beta mn$ functionally
independent variables
$$
  (d\mathbf{X}) = \bigwedge_{i = 1}^{m}\bigwedge_{j = 1}^{n}dx_{ij} \quad \mbox{ where }
    \quad dx_{ij} = \bigwedge_{k = 1}^{\beta}dx_{ij}^{(k)}.
$$
If $\mathbf{S} \in \mathfrak{S}_{m}^{\beta}$ (or $\mathbf{S} \in \mathfrak{T}_{L}^{\beta}(m)$
is a lower triangular matrix) then $(d\mathbf{S})$ (the Lebesgue measure in
$\mathfrak{S}_{m}^{\beta}$ or in $\mathfrak{T}_{L}^{\beta}(m)$, the set of lower triangular
matrices) denotes the exterior product of the $m(m+1)\beta/2$ functionally independent
variables (or denotes the exterior product of the $m(m-1)\beta/2 + n$ functionally independent
variables, if $s_{ii} \in \Re$ for all $i = 1, \dots, m$)
$$
  (d\mathbf{S}) = \left\{
                    \begin{array}{ll}
                      \displaystyle\bigwedge_{i \leq j}^{m}\bigwedge_{k = 1}^{\beta}ds_{ij}^{(k)}, &  \\
                      \displaystyle\bigwedge_{i=1}^{m} ds_{ii}\bigwedge_{i < j}^{m}\bigwedge_{k = 1}^{\beta}
                      ds_{ij}^{(k)}, & \hbox{if } s_{ii} \in \Re.
                    \end{array}
                  \right.
$$
The context generally establishes the conditions on the elements of $\mathbf{S}$, that is, if
$s_{ij} \in \Re$, $\in \mathfrak{C}$, $\in \mathfrak{H}$ or $ \in \mathfrak{O}$. It is
considered that
$$
  (d\mathbf{S}) = \bigwedge_{i \leq j}^{m}\bigwedge_{k = 1}^{\beta}ds_{ij}^{(k)}
   \equiv \bigwedge_{i=1}^{m} ds_{ii}\bigwedge_{i < j}^{m}\bigwedge_{k =
1}^{\beta}ds_{ij}^{(k)}.
$$
Observe, too, that for the Lebesgue measure $(d\mathbf{S})$ defined thus, it is required that
$\mathbf{S} \in \mathfrak{P}_{m}^{\beta}$, that is, $\mathbf{S}$ must be a non singular
Hermitian matrix (Hermitian definite positive matrix).

If $\mathbf{\Lambda} \in \mathfrak{D}_{m}^{\beta}$ then $(d\mathbf{\Lambda})$ (the Legesgue
measure in $\mathfrak{D}_{m}^{\beta}$) denotes the exterior product of the $\beta m$
functionally independent variables
$$
  (d\mathbf{\Lambda}) = \bigwedge_{i = 1}^{n}\bigwedge_{k = 1}^{\beta}d\lambda_{i}^{(k)}.
$$
If $\mathbf{H}_{1} \in \mathcal{V}_{m,n}^{\beta}$ then
$$
  (\mathbf{H}^{*}_{1}d\mathbf{H}_{1}) = \bigwedge_{i=1}^{m} \bigwedge_{j =i+1}^{n}
  \mathbf{h}_{j}^{*}d\mathbf{h}_{i}.
$$
where $\mathbf{H} = (\mathbf{H}^{*}_{1}|\mathbf{H}^{*}_{2})^{*} = (\mathbf{h}_{1}, \dots,
\mathbf{h}_{m}|\mathbf{h}_{m+1}, \dots, \mathbf{h}_{n})^{*} \in \mathfrak{U}^{\beta}(n)$. It
can be proved that this differential form does not depend on the choice of the $\mathbf{H}_{2}$
matrix. When $n = 1$; $\mathcal{V}^{\beta}_{m,1}$ defines the unit sphere in
$\mathfrak{F}^{m}$. This is, of course, an $(m-1)\beta$- dimensional surface in
$\mathfrak{F}^{m}$. When $n = m$ and denoting $\mathbf{H}_{1}$ by $\mathbf{H}$,
$(\mathbf{H}d\mathbf{H}^{*})$ is termed the \emph{Haar measure} on $\mathfrak{U}^{\beta}(m)$.

The surface area or volume of the Stiefel manifold $\mathcal{V}^{\beta}_{m,n}$ is
\begin{equation}\label{vol}
    \Vol(\mathcal{V}^{\beta}_{m,n}) = \int_{\mathbf{H}_{1} \in
  \mathcal{V}^{\beta}_{m,n}} (\mathbf{H}_{1}d\mathbf{H}^{*}_{1}) =
  \frac{2^{m}\pi^{mn\beta/2}}{\Gamma^{\beta}_{m}[n\beta/2]},
\end{equation}
where $\Gamma^{\beta}_{m}[a]$ denotes the multivariate \emph{Gamma function} for the space
$\mathfrak{S}_{m}^{\beta}$, and is defined by
\begin{eqnarray*}
  \Gamma_{m}^{\beta}[a] &=& \displaystyle\int_{\mathbf{A} \in \mathfrak{P}_{m}^{\beta}}
  \etr\{-\mathbf{A}\} |\mathbf{A}|^{a-(m-1)\beta/2 - 1}(d\mathbf{A}) \\
    &=& \pi^{m(m-1)\beta/4}\displaystyle\prod_{i=1}^{m} \Gamma[a-(i-1)\beta/2],
\end{eqnarray*}
where $\etr(\cdot) = \exp(\tr(\cdot))$, $|\cdot|$ denotes the determinant and $\re(a)
> (m-1)\beta/2$, see \citet{gr:87}. Similarly, from \citet{h:55} the \emph{multivariate beta function} for the
space $\mathfrak{S}^{\beta}_{m}$, can be defined as
\begin{eqnarray}
    \mathcal{B}_{m}^{\beta}[b,a] &=& \int_{\mathbf{0}<\mathbf{B}<\mathbf{I}_{m}}
    |\mathbf{B}|^{a-(m-1)\beta/2-1} |\mathbf{I}_{m} - \mathbf{B}|^{b-(m+1)\beta/2-1}
    (d\mathbf{B}) \nonumber\\
    &=& \int_{\mathbf{A} \in \mathfrak{P}_{m}^{\beta}} |\mathbf{A}|^{a-(m-1)\beta/2-1}
    |\mathbf{I}_{m} + \mathbf{A}|^{-(a+b)} (d\mathbf{A}) \nonumber\\ \label{beta}
    &=& \frac{\Gamma_{m}^{\beta}[a] \Gamma_{m}^{\beta}[b]}{\Gamma_{m}^{\beta}[a+b]},
\end{eqnarray}
where $\mathbf{A} = (\mathbf{I}-\mathbf{B})^{-1} -\mathbf{I}$, Re$(a) > (m-1)\beta/2$ and
Re$(b)> (m-1)\beta/2$.

Now, we show two Jacobians in terms of the $\beta$ parameter, which are based on the work of
\citet{k:84} and \citet{d:02}. These results are proposed as extensions of real, complex or
quaternion cases, see  \citet{j:64}, \citet{k:65}, \citet{me:91}, \citet{rva:05a} and
\citet{lx:09}.

\begin{lem}\label{lemlt}
Let $\mathbf{X}$ and $\mathbf{Y} \in {\mathcal L}_{m,n}^{\beta}$, and let $\mathbf{Y} =
\mathbf{AXB} + \mathbf{C}$, where $\mathbf{A} \in {\mathcal L}_{m,m}^{\beta}$, $\mathbf{B} \in
{\mathcal L}_{n,n}^{\beta}$ and $\mathbf{C} \in {\mathcal L}_{m,n}^{\beta}$ are constant
matrices. Then
\begin{equation}\label{lt}
    (d\mathbf{Y}) = |\mathbf{A}^{*}\mathbf{A}|^{\beta n/2} |\mathbf{B}^{*}\mathbf{B}|^{\beta
    m/2}(d\mathbf{X}).
\end{equation}
\end{lem}

\begin{lem}\label{lemhlt}
Let $\mathbf{X}$ and $\mathbf{Y} \in \mathfrak{P}_{m}^{\beta}$, and let $\mathbf{Y} =
\mathbf{AXA^{*}} + \mathbf{C}$, where $\mathbf{A}$ and $\mathbf{C} \in {\mathcal
L}_{m,m}^{\beta}$ are constant matrices. Then
\begin{equation}\label{hlt}
    (d\mathbf{Y}) = |\mathbf{A}^{*}\mathbf{A}|^{\beta(m-1)/2+1} (d\mathbf{X}).
\end{equation}
\end{lem}

\begin{lem}[Singular value decomposition, $SVD$]\label{lemsvd}
Let $\mathbf{X} \in {\mathcal L}_{m,n}^{\beta}$, such that $\mathbf{X} =
\mathbf{V}_{1}\mathbf{DW}^{*}$ with $\mathbf{V}_{1} \in {\mathcal V}_{m,n}^{\beta}$,
$\mathbf{W} \in \mathfrak{U}^{\beta}(m)$ and $\mathbf{D} = \diag(d_{1}, \cdots,d_{m}) \in
\mathfrak{D}_{m}^{1}$, $d_{1}> \cdots > d_{m} > 0$. Then
\begin{equation}\label{svd}
    (d\mathbf{X}) = 2^{-m}\pi^{\tau} \prod_{i = 1}^{m} d_{i}^{\beta(n - m + 1) -1}
    \prod_{i < j}^{m}(d_{i}^{2} - d_{j}^{2})^{\beta} (d\mathbf{D}) (\mathbf{V}_{1}d\mathbf{V}_{1}^{*})
    (\mathbf{W}d\mathbf{W}^{*}),
\end{equation}
where
$$
  \tau = \left\{
             \begin{array}{rl}
               0, & \beta = 1; \\
               -m, & \beta = 2; \\
               -2m, & \beta = 4; \\
               -4m, & \beta = 8.
             \end{array}
           \right.
$$
\end{lem}

\begin{lem}\label{lemW}
Let $\mathbf{X} \in {\mathcal L}_{m,n}^{\beta}$, and  $\mathbf{S} = \mathbf{X}\mathbf{X}^{*}
\in \mathfrak{P}_{m}^{\beta}.$ Then
\begin{equation}\label{w}
    (d\mathbf{X}) = 2^{-m} |\mathbf{S}|^{\beta(n - m + 1)/2 - 1}
    (d\mathbf{S})(\mathbf{V}_{1}d\mathbf{V}_{1}^{*}).
\end{equation}
\end{lem}
\end{document}